\documentclass[10pt]{amsart}
\usepackage[utf8]{inputenc}
\usepackage[T1]{fontenc}
\usepackage[english]{babel}
\title[Exotic LLT at the phase transition in free products]{Exotic local limit theorems at the phase transition in free products}
\author{Matthieu Dussaule and Marc Peign\'e and Samuel Tapie}
\date{}

\usepackage{comment}
\usepackage{amsmath}
\usepackage{amssymb}
\usepackage{dsfont}
\usepackage{amsfonts}
\usepackage{amsthm}
\usepackage{tikz}
\usepackage{graphicx}
\usepackage{fancyhdr}
\usepackage{caption}
\usepackage{enumerate}
\usepackage{mathtools}
\usepackage[colorlinks=true,pdfstartview=FitH,bookmarks=false]{hyperref}
\hypersetup{urlcolor=black,linkcolor=black,citecolor=black,colorlinks=true}

\newcommand{\vertiii}[1]{{\left\vert\kern-0.25ex\left\vert\kern-0.25ex\left\vert #1 
    \right\vert\kern-0.25ex\right\vert\kern-0.25ex\right\vert}}

\newcommand\Z{\mathbb{Z}}

\theoremstyle{plain}
\newtheorem{definition}{Definition}[section]
\newtheorem{proposition}[definition]{Proposition}
\newtheorem{corollary}[definition]{Corollary}
\newtheorem{theorem}[definition]{Theorem}
\newtheorem{lemma}[definition]{Lemma}
\newtheorem*{prop*}{Proposition}
\newtheorem*{lem*}{Lemma}
\newtheorem{wlemma}[definition]{Erroneous Lemma}

\theoremstyle{remark}

\newtheorem*{rem*}{Remark}

\usepackage{etoolbox}
\apptocmd{\sloppy}{\hbadness 10000\relax}{}{}
\apptocmd{\sloppy}{\vbadness 10000\relax}{}{}

 \numberwithin{equation}{section}

\begin{document}

\begin{abstract}
We construct random walks on free products of the form $\Z^3*\Z^d$, with $d=5$ or $6$ which are divergent and not spectrally positive recurrent.
We then derive a local limit theorem for these random walks, proving that $\mu^{*n}(e)\sim CR^{-n}n^{-5/3}$ if $d=5$ and $\mu^{*n}(e)\sim CR^{-n}n^{-3/2}\log (n)^{-1/2}$ if $d=6$,
where $\mu^{*n}$ is the $n$th convolution power of $\mu$ and $R$ is the inverse of the spectral radius of $\mu$.
This disproves a result of Candellero and Gilch \cite{CandelleroGilch} and a result of the authors of this paper that was stated in a first version of \cite{DPT}.
This also shows that the classification of local limit theorems on free products of the form $\Z^{d_1}*\Z^{d_2}$ or more generally on relatively hyperbolic groups with respect to virtually abelian subgroups is incomplete.
\end{abstract}

\maketitle

\section{Introduction}
Let $\Gamma$ be a finitely generated group and let $\mu$ be a probability measure on $\Gamma$.
In the sequel, we will always assume that $\mu$ is finitely supported and symmetric.
Denote by $\mu^{*n}$ the $n$th convolution power of $\mu$, defined by
$$\mu^{*n}(x)=\sum_{y_1,...,y_{n-1}\in \Gamma}\mu(y_1)\mu(y_1^{-1}y_2)...\mu(y_{n-1}^{-1}x).$$
Let us consider the random walk $(X_n)_n$ driven by $\mu$, defined by $X_n=g_1...g_n$ where $g_k$ are independent random variables whose distribution are given by $\mu$.
Then, $\mu^{*n}$ is the $n$th step distribution of the random walk, so
for all $x\in \Gamma$, $\mu^{*n}(x)$ is the probability that $X_n=x$.

We will also always assume that the random walk is \emph{admissible}, i.e.\ for every $x\in \Gamma$, there exists $n$ such that $\mu^{*n}(x)$ is positive.
In other words, every element of the group can be visited with positive probability, i.e.\ the support of $\mu$ generates $\Gamma$ as a semi-group (hence as a group, since $\mu$ is symmetric).
We also say that the measure $\mu$ is admissible.

We denote by $\rho$ the spectral radius of the random walk defined by
$$\rho=\limsup \mu^{*n}(x)^{1/n}.$$
The spectral radius $\rho$ belongs to $[0,1]$ and is independent of $x$, provided that $\mu$ is admissible, see \cite[(1.8)]{Woessbook}.

\medskip
The local limit problem consists in finding the asymptotic behavior of $\mu^{*n}(x)$ as $n$ goes to infinity.
We assume for simplicity that $\mu$ is aperiodic, i.e.\ there exists $n_0$ such that for every $n\geq n_0$, $\mu^{*n}(e)>0$, where $e$ is the identity element of $\Gamma$.
In many cases, the asymptotics arising in local limit theorems are of the form
\begin{equation}\label{nonexoticLLT}
    \mu^{*n}(x)\sim C_xR^{-n}n^{-\alpha},
\end{equation}
where $R$ is the inverse of the spectral radius.
This is for example the case in all abelian groups of rank $d$, with $\alpha=d/2$, see \cite[Theorem~13.12]{Woessbook} and references therein, and more generally in all nilpotent groups of homogeneous dimension $D$, with $\alpha=D/2$, see \cite[Corollary~1.17]{Alexopoulos}.
This is also the case in all hyperbolic groups with $\alpha=3/2$, see \cite{GerlWoess}, \cite{Lalley} for the case of trees and \cite{GouezelLLT} for the general case.
Finally, to our knowledge, this was also the case so far in all known examples of relatively hyperbolic groups.

In the context of free products of the form $\Gamma=\Z^{d_1}*\Z^{d_2}$, Candellero and Gilch \cite{CandelleroGilch} gave a complete classification of every possible local limit theorem.
In particular, they proved that they always are of the form~(\ref{nonexoticLLT}), with $\alpha=3/2$ or $\alpha=d_i/2$ and the latter case can only happen if $d_i\geq 5$.
Although in this paper we will not work in the general setting of relatively hyperbolic groups, let us mention that free products are the simplest examples of such groups and results of \cite{CandelleroGilch} are being generalized to this setting in recent works by the authors, see \cite{DussauleLLT1}, \cite{DussauleLLT2}, \cite{DPT}.

\medskip
Our main goal in this note is to disprove \cite[Lemma~4.5]{CandelleroGilch} and a similar statement that appeared in a first version of \cite{DPT}.
In particular, we prove that the classification obtained in \cite{CandelleroGilch} is incomplete: we derive a local limit theorem on $\Z^3*\Z^5$ of the form~(\ref{nonexoticLLT}) but with unexpected exponent $\alpha=5/3$, and a local limit theorem on $\Z^3*\Z^6$ which is \emph{not} of the form~(\ref{nonexoticLLT}).
Before stating our main results, let us introduce some terminology.

\medskip
We consider the Green function $G(x,y|r)$ defined by
$$G(x,y|r)=\sum_{n\geq0}\mu^{*n}(x^{-1}y)r^n.$$
If $x=y=e$, we will often write $G(e,e|r)=G(r)$.
Its radius of convergence $R$ is independent of $x$ and $y$, provided $\mu$ is admissible and it is the inverse of the spectral radius of $\mu$.
All the groups under consideration in this paper will be non-amenable.
Consequently,
\begin{itemize}
    \item by a landmark result of Kesten \cite{Kesten}, $R>1$ (see also \cite[Corollary~12.5]{Woessbook}),
    \item  by a result of Guivarc'h \cite{GUivarch}, $G(R)$ is finite (see also \cite[Theorem~7.8]{Woessbook}).
\end{itemize}

\medskip
Following the notations of \cite{DussauleLLT1}, we define $I^{(k)}(r)$ by
\begin{equation}\label{defnI}
I^{(k)}(r)=\sum_{x_1,...,x_k\in \Gamma}G(e,x_1|r)G(x_1,x_2|r)...G(x_{k-1},x_k|r)G(x_k,e|r).
\end{equation}
The sums $I^{(k)}$ are related to the $k$th derivatives of the Green function.
Precisely, by \cite[Proposition~1.9]{GouezelLalley}, $I^{(1)}(r)=rG'(r)+G(r)$ and similar formulae hold for higher derivatives.
Following \cite{DussauleLLT1}, we say that the random walk driven by $\mu$ is \textit{divergent} if $I^{(1)}(R)$, or equivalently $G'(R)$, is infinite.
We say that it is \textit{convergent} otherwise.

Assume from now on that $\Gamma=\Gamma_1*\Gamma_2$. 
We  define for $i=1,2$
\begin{equation}\label{defnIki}
    I^{(k)}_{i}(r)=\sum_{x_1,...,x_k\in \Gamma_i}G(e,x_1|r)G(x_1,x_2|r)...G(x_{k-1},x_k|r)G(x_k,e|r)
\end{equation}
and we set
\begin{equation}\label{defnJ}
J^{(k)}(r)=I^{(k)}_{1}(r)+I^{(k)}_{2}(r).
\end{equation}
Still following \cite{DussauleLLT1}, we say that the random walk driven by $\mu$ is\textit{ spectrally positive recurrent} if it is divergent and if $J^{(2)}(R)$ is finite.

\medskip
For $i=1,2$, we also consider the first return kernel $p_{\Gamma_i,r}$ to $\Gamma_i$ associated with $r\mu$ (see~(\ref{defnfirstreturn}) for a proper definition).
Then, $p_{\Gamma_i,r}$ defines a transition kernel on $\Gamma_i$ and we denote by $R_i(r)$ the inverse of its spectral radius.
We say that the random walk driven by $\mu$ is \textit{spectrally degenerate} along $\Gamma_i$ if $R_i(R)=1$.
When both $R_1(R)$ and $R_2(R)$ are bigger than 1, we say that the random walk is \textit{spectrally non-degenerate}.

Roughly speaking, when the random walk is spectrally degenerate along $\Gamma_i$, the free factor $\Gamma_i$ has strong influence on its asymptotic behavior;
we refer to \cite{DussauleLLT2} and \cite{DG21} for further details.
This notion should be compared with what is called "typical case" in \cite{Woessbook}, where another way of measuring influence of a free factor is given.
By \cite[Proposition~2.9]{DussauleLLT2}, these two notions coincide, i.e.\  the "typical case" corresponds to the case of a spectrally non-degenerate random walk.

\medskip
All these quantities and definitions can be generalized to the context of relatively hyperbolic groups, replacing free factors with the appropriate notion of maximal parabolic subgroups.
The current classification of local limit theorems on relatively hyperbolic groups is as follows.
When the random walk is spectrally non-degenerate, the local limit has the form~(\ref{nonexoticLLT}), with $\alpha=3/2$ \cite{DussauleLLT2}.
This was first proved by Woess \cite{WoessLLT} for random walks on free products in the "typical case" situation.
When the random walk is spectrally positive recurrent, we can only prove the rough estimate $\mu^{*n}(e)\asymp R^{-n}n^{-3/2}$, which means that the ratio of the quantities on the left and right hand-side is bounded away from 0 and infinity \cite{DussauleLLT1}.
When the random walk is convergent and parabolic subgroups are virtually abelian, the local limit theorem has the form~(\ref{nonexoticLLT}), with $\alpha=d/2$, where $d$ is the minimal rank of a parabolic subgroup along which the random walk is spectrally degenerate \cite{DPT}.
Moreover, in this situation, one can only have $d\geq 5$.

Thus, we recover so far the classification given in \cite{CandelleroGilch} and presented above.
Furthermore, up to the present paper, for free products of the form $\Z^{d_1}*\Z^{d_2}$ the case of a divergent and not spectrally positive recurrent random walk was considered as not being able to occur, see \cite[Lemma~4.5]{CandelleroGilch}.
As announced, we disprove here this result and we actually construct such a random walk on $\Gamma=\Z^3*\Z^d$, with $d=5$ or $6$.
As a consequence, the classification of possible behaviors of $\mu^{*n}$ needs to be completed.
We also derive a local limit theorem for the random walk we construct.
This is the first step into this program.

\begin{theorem}\label{maintheoremconstruction}
Let $\Gamma=\Gamma_1*\Gamma_2$, with $\Gamma_1=\Z^3$, $\Gamma_2=\Z^d$ and $d\in \{5,6\}$.
For $i=1,2$, let $\mu_i$ be a finitely supported, admissible and symmetric probability measure on $\Gamma_i$.
For $\alpha\in [0,1]$, let $\mu_\alpha$ be the probability measure $\mu_\alpha=\alpha\mu_1+(1-\alpha)\mu_2$ on $\Gamma$.
Then, there exists $\alpha_*\in (0,1)$ such that the random walk driven by $\mu_{\alpha_*}$ is divergent and not spectrally positive recurrent.
\end{theorem}

When either $\mu_1$ or $\mu_2$ is aperiodic, the same property holds for $\mu_\alpha$.
From now on, in order to simplify the argument, we assume that both measures $\mu_1$ and $\mu_2$ are aperiodic; this allows us to avoid to consider several sub-cases for the estimation of the Green functions associated with the corresponding random walks on $\Gamma_1$ and $\Gamma_2$, see~(\ref{asymptoticG''oddcase}) and~(\ref{asymptoticG''evencase}) below.

\begin{theorem}\label{maintheoremLLT}
 Assume that the measures $\mu_1$ and $\mu_2$ are aperiodic.
 Then, the random walk on $\Gamma$ driven by $\mu_{\alpha_*}$ given by Theorem~\ref{maintheoremconstruction}
satisfies the following local limit theorem: 
If $d=5$, we have
$$\mu_{\alpha_*}^{*n}(e)\sim CR^{-n}n^{-5/3}$$
and if $d=6$, we have
$$\mu_{\alpha_*}^{*n}(e)\sim C R^{-n}n^{-3/2}\log(n)^{-1/2},$$
where $R$ is the inverse of the spectral radius of $\mu_{\alpha_*}$.
\end{theorem}
Without assuming aperiodicity, the same asymptotics hold for $\mu^{*2n}$, since $\mu$ is symmetric, so its period must be 1 or 2.

\medskip
Let us state that to our knowledge, the asymptotic for $d=6$ in Theorem~\ref{maintheoremLLT} gives the first example of a local limit theorem on a non-amenable group which is not of the form~(\ref{nonexoticLLT}).
For amenable groups, the situation is quite different and there exist many examples where $\mu^{*n}(e)$ behaves like $\mathrm{exp}(-n^c)$.

 Up to a sub-exponential error term, this is the case for all polycyclic groups of exponential growth \cite{V91}, \cite[Theorem~1]{Alexopoulospolycyclic} and for amenable Baumslag-Solitar groups \cite[Theorem~5.2~(5.2)]{CGP}, with $c=1/3$.
This is also the case for lamplighter groups of the form $A\wr \Z^d$, where $A$ is a finite non-trivial group \cite[Theorem~5.2~(5.6)]{CGP}, with $c=d/(d+2)$.
Note that amenable Baumslag-Solitar groups and lamplighter groups are examples of solvable non-polycyclic groups.

For $d=1$, a precise local limit theorem for the lamplighter group of the form $\mu^{*n}(e)\sim C n^{1/6}\mathrm{exp}(-n^{1/3})$ was proven by Revelle \cite{Revelle}.
This was further extended to Diestel-Leader graphs $\mathrm{DL}(q,r)$ by Bartholdi, Neuhauser and Woess, see \cite[Theorem~5.4]{BW} and \cite[Corollary~5.26]{BNW}.
Diestel-Leader graphs are not amenable when $q\neq r$, since the spectral radius of the simple random walk is smaller than 1, see \cite[(1.3)]{BW}.
Thus, the examples of \cite{BW} and \cite{BNW} already  provide local limit theorems which are not of the form~(\ref{nonexoticLLT}) but of the form
$\mu^{*n}(e)\sim CR^{-n}\mathrm{exp}(-n^c)n^{\alpha}$ for non-amenable graphs.
However, according to \cite[Theorem~1.4]{EFW} when $q\neq r$, $\mathrm{DL}(q,r)$ is not quasi-isometric to the Cayley graph of a finitely generated group.

We also refer to \cite{BZ} where many other examples are given, beyond the class of amenable groups.
Asymptotics are only given there for $-\log \mu^{*n}(e)$ though. Thus, for non-amenable groups, these examples only recover the fact that $R>1$.

\medskip
We now briefly outline the content of our paper.
In Section~\ref{Sectioncharacterization}, we give various characterizations of spectral degeneracy in terms of quantities that are suited to the study of random walks on free products.
Along the way, we introduce functions and quantities defined in \cite{Woessbook}.
The conclusion of this section is a useful characterization of spectral degeneracy and divergence in terms of the sign of a single quantity, see precisely Corollary~\ref{trichotomy}.

In Section~\ref{sectionconstruction}, we use Corollary~\ref{trichotomy} to prove Theorem~\ref{maintheoremconstruction}, that is, we construct a probability measure $\mu$ on $\Z^3*\Z^d$, $d=5$ or $6$, which is divergent but not spectrally positive recurrent.
We will actually construct a family of probability measure $\mu_\alpha$ and exhibit a phase transition at some $\alpha_*$.
The measure $\mu_{\alpha_*}$ will have the required properties.

Finally, Section~\ref{SectionLLT} is devoted to derive a local limit theorem for $\mu_{\alpha_*}$, thus proving Theorem~\ref{maintheoremLLT}.
This is done by first finding precise asymptotics of the derivative of the Green function $G_r(e,e)$ as $r\to R$ and then using Karamata's Tauberian theorem.
Most of the intermediate results in this section are of geometric nature and we believe it should be possible to extend them to relatively hyperbolic groups, with (possibly challenging) new arguments replacing those that rely on the combinatorial structure of free products.

\section{Characterizations of spectral degeneracy in free products}\label{Sectioncharacterization}

Let $\Gamma=\Gamma_1*\Gamma_2$ be a free product of two groups.
%Then, $\Gamma$ is hyperbolic relative to the conjugates of the free factors $\Gamma_1$ and $\Gamma_2$.
Consider finitely supported, symmetric and admissible probability measures $\mu_1$ and $\mu_2$ on $\Gamma_1$ and $\Gamma_2$ respectively.
For $\alpha\in [0,1]$, set
$$\mu_\alpha=\alpha\mu_1+(1-\alpha)\mu_2.$$
In the sequel, we write $\mu$ for $\mu_\alpha$ and we set $\alpha_1=\alpha$ and $\alpha_2=(1-\alpha)$.
If $\alpha_i>0$, the probability measure $\mu$ is finitely supported, symmetric and admissible on $\Gamma$.
Such a probability measure is called adapted to the free product structure.
We denote by $R$ the inverse of the spectral radius of $\mu$ and by $R_i$ the inverse of the spectral radius of $\mu_i$.

The Green functions $G$, $G_1$ and $G_2$ of $\mu$, $\mu_1$ and $\mu_2$ respectively are related as follows.
For $i=1,2$, for every $x,y\in \Gamma_i$, for every $r\leq R$,
\begin{equation}\label{formulaWoess}
\frac{G(x,y|r)}{G(e,e|r)}=\frac{G_{i}(x,y|\zeta_i(r))}{G_{i}(e,e|\zeta_i(r))},
\end{equation}
where $\zeta_i$ is a continuous function of $r$, see \cite[Proposition~9.18]{Woessbook} for an explicit formula.
We always have $\zeta_i(R)\leq R_i$ and for $r<R$, $\zeta_i(r)<R_i$.

\medskip
We denote by $p_{\Gamma_i,r}$ the first return transition kernel to $\Gamma_i$ associated with $r\mu$, which is defined as
\begin{equation}\label{defnfirstreturn}
p_{\Gamma_i,r}(x,y)=\sum_{n\geq 1}\hspace{.1cm}\sum_{z_1,...,z_n\notin \Gamma_i}r^n\mu(x^{-1}z_1)\mu(z_1^{-1}z_2)...\mu(z_n^{-1}y).
\end{equation}
We denote by $G_{\Gamma_i,r}$ the Green function associated with $p_{\Gamma_i,r}$.
By \cite[Lemma~4.4]{DG21}, for every $x,y\in \Gamma_i$, it holds
\begin{equation}\label{Greencoincide}
G_{\Gamma_i,r}(x,y|1)=G(x,y|r),
\end{equation}
which is actually the main reason for introducing $p_{\Gamma_i,r}$.

In fact, because $\mu$ is adapted to the free product structure, if the random walk ever leaves $\Gamma_i$ at some point $x$, it can only come back to $\Gamma_i$ at the same point $x$.
We deduce that the first return kernel $p_{\Gamma_i,r}$ can be written in our context as
$$p_{\Gamma_i,r}(e,x)=\alpha_ir\mu_i(x)+w_i\delta_{e,x},$$
where $w_i=w_i(r)$ is the weight of the first return to $e$ associated to $r\mu$, starting with a step driven by $\alpha_j\mu_j$, $j\neq i$.
Thus, \cite[Lemma~9.2]{Woessbook} shows that for any $x,y\in \Gamma_i$,
\begin{equation}\label{superWoess}
G_{\Gamma_i,r}(x,y|t)=\frac{1}{1-w_it}G_{i}\left (x,y\ \bigg|\ \frac{\alpha_irt}{1-w_it}\right ).
\end{equation}
In particular, for $t=1$,
$$G_{\Gamma_i,r}(x,y|1)=\frac{1}{1-w_i}G_{i}\left (x,y\ \bigg|\ \frac{\alpha_ir}{1-w_i}\right )$$
Since $G_{\Gamma_i,r}(x,y|1)=G(x,y|r)$, we recover~(\ref{formulaWoess}) with
$$\zeta_i(r)=\frac{\alpha_ir}{1-w_i}.$$

Recall that following \cite{DG21}, we say that the random walk is spectrally degenerate along $\Gamma_i$ if the spectral radius of the first return kernel $p_{\Gamma_i,R}$ is 1.
In this section, we prove equivalent conditions to spectral degeneracy, using the more standard terminology for free products introduced in \cite[Chapter~9]{Woessbook}.

The following characterization is proved in \cite[Proposition~2.9]{DussauleLLT2}. We detail it here for convenience.

\begin{lemma}\label{firstequivalence}
The random walk is spectrally degenerate along $\Gamma_i$ if and only if $\zeta_i(R)=R_i$.
\end{lemma}

\begin{proof}
By applying~(\ref{superWoess}) with $t=1+\epsilon$ and $r=R$, we get
$$G_{\Gamma_i,R}(x,y|1+\epsilon)=\frac{1}{1-w_i(1+\epsilon)}G_{i}\left (x,y\ \bigg|\ \frac{\alpha_iR(1+\epsilon)}{1-w_i(1+\epsilon)}\right ).$$
The condition $\epsilon>0$ yields $\frac{\alpha_iR(1+\epsilon)}{1-w_i(1+\epsilon)}>\frac{\alpha_iR}{1-w_i}=\zeta_i(R)$.
Thus, there exists $t>1$ such that $G_{\Gamma_i,R}(x,y|t)$ is finite if and only if there exists $z>\zeta_i(R)$ such that $G_{i}(x,y|z)$ is finite, which concludes the proof.
\end{proof}

In \cite{Woessbook}, the situation where $\zeta_i(R)<R_i$ for $i=1,2$ is called the "typical case".
Thus, Lemma~\ref{firstequivalence} shows that this typical case corresponds to being spectrally non-degenerate.

\begin{lemma}\label{expressionzetaGzeta}
For all $r\leq R$, we have $\zeta_i(r)G_i(\zeta_i(r))=\alpha_irG(r)$.
\end{lemma}

\begin{proof}
Let us introduce the quantities $U$ and $U_i$, $i=1,2$, defined by
$$
U(r)= U(e,e|r) = \sum_{n\geq 0} \mathbb P((\text{first return time of the }\mu\text{-random walk to } e) = n) r^n
$$
and
$$
U_i(r)= U_i(e,e|r) = \sum_{n\geq 0} \mathbb P((\text{first return time of the }\mu_i\text{-random walk to } e) = n) r^n.
$$
By \cite[Lemma~1.13~(a)]{Woessbook},
$G(r)(1-U(r))=G_i(r)(1-U_i(r))=1$.
Following \cite[Proposition~9.18~(b)]{Woessbook}, the weight $w_i$ may be written as $w_i=U(r)-H_i(r)$, where $H_i$ satisfies the equation
$$\frac{G(r)H_i(r)}{1+G(r)H_i(r)}=U_i(\zeta_i(r)),$$
i.e.
$G(r)H_i(r)(1-U_i(\zeta_i(r)))=U_i(\zeta_i(r))$.
The equality $G_i(1-U_i)=1$ yields
$$\frac{G(r)H_i(r)}{G_i(\zeta_i(r))}=\frac{G_i(\zeta_i(r))-1}{G_i(\zeta_i(r))},$$
i.e.
\begin{equation}\label{equationHi}
1+G(r)H_i(r)=G_i(\zeta_i(r)).
\end{equation}
Consequently, since $G(1-U)=1$,
$$\zeta_i(r)=\frac{\alpha_ir}{1-w_i}=\frac{\alpha_ir}{1-U(r)+H_i(r)}=\frac{\alpha_irG(r)}{1+G(r)H_i(r)},$$
so by~(\ref{equationHi}),
\begin{equation*}
\zeta_i(r)=\frac{\alpha_irG(r)}{G_i(\zeta_i(r))}. \qedhere
\end{equation*}
\end{proof}

Let us now introduce the notations $\theta=RG(R)$, $\theta_i=R_iG(R_i)$, $i=1,2$ and $\overline{\theta}=\min \{\theta_i/\alpha_i\}$.
These parameters play a crucial role in the study of the Green function on free products in \cite[Chapter~9]{Woessbook}.
In particular, by \cite[Theorem~9.19]{Woessbook}, it holds that $\theta\leq \theta_i/\alpha_i$, $i=1,2$, so $\theta\leq \overline{\theta}$.

The following statement gives a characterization of spectral degeneracy in terms of $\theta$ and $\overline{\theta}$

\begin{lemma}\label{secondequivalence}
The random walk is spectrally degenerate along $\Gamma_i$ if and only if $\overline{\theta}=\theta=\theta_i/\alpha_i$.
\end{lemma}

\begin{proof}
Assume that $\theta=\overline{\theta}=\theta_i/\alpha_i$, i.e.\ $RG(R)=R_iG_i(R_i)/\alpha_i$.
By Lemma~\ref{expressionzetaGzeta}, we thus have
$$\zeta_i(R)=R_i\frac{G_i(R_i)}{G_i(\zeta_i(R))}.$$
Since $\zeta_i(R)\leq R_i$, we deduce that $G_i(R_i)\geq G_i(\zeta_i(R))$ and so $\zeta_i(R)\geq R_i$.
Finally, $\zeta_i(R)=R_i$ and so the random walk is spectrally degenerate along $\Gamma_i$ by Lemma~\ref{firstequivalence}.

Conversely, if the random walk is spectrally degenerate along $\Gamma_i$, then $\zeta_i(R)=R_i$ by Lemma~\ref{firstequivalence}.
This implies $G_i(\zeta_i(R))=G_i(R_i)$.
Consequently,
$$R_i=\zeta_i(R)=\frac{\alpha_iRG(R)}{G_i(\zeta_i(R))}=\frac{\alpha_iRG(R)}{G_i(R_i)}.$$
Therefore, $R_iG_i(R_i)/\alpha_i=RG(R)$, i.e.\ $\theta=\theta_i/\alpha_i$.
Combining this with the inequality $\theta\leq \overline{\theta}$,  we finally obtain $\theta=\overline{\theta}$.
\end{proof}

Following \cite[Chapter~9]{Woessbook}, let us introduce two functions $\Phi$ and $\Psi$ which are very useful in the context of free products.

On the one hand, he function $\Phi$ is defined implicitly by the formula
$$G(r)=\Phi(rG(r))$$
for every $r\leq R$.
This function is defined in general on an open neighborhood (inside the complex plane) of the interval $[0,\theta)$.
Since $G(R)$ is finite, it is also defined on $[0,\theta]$.

On the other hand, we set $\Psi(t)=\Phi(t)-t\Phi'(t)$.
By \cite[(9.14)]{Woessbook}, letting $t=rG(r)$, we have
\begin{align*}
    \Psi(t)&=\frac{1}{rU'(r)+1-U(r)}\\
    &=\frac{1}{1+\sum_{n\geq 0}(n-1)\mathbb P((\text{first return time of the }\mu\text{-random walk to } e) = n)r^n}.
\end{align*}
In particular $\Psi$ is strictly decreasing on the interval $[0,\theta]$ and satisfies $\Psi(t)<1$ for $t>0$ and $\Psi(0)=1$.
The equality $G(r)(1-U(r))=1$ readily implies that
\begin{equation}\label{equationPsiG}
\Psi(t)=\frac{G(r)^2}{rG(r)'+G(r)}.
\end{equation}
Thus, $\Psi(\theta)=0$ if and only if $G'(R)=\infty$, since $G(R)$ is finite.

In our context of free products $\Gamma=\Gamma_1*\Gamma_2$,
we have by \cite[Theorem~9.19]{Woessbook}
\begin{equation}\label{equationPhi}
\Phi(t)=\Phi_1(\alpha_1t)+\Phi_2(\alpha_2t)-1
\end{equation}
and
\begin{equation}\label{equationPsi}
\Psi(t)=\Psi_1(\alpha_1t)+\Psi_2(\alpha_2t)-1.
\end{equation}

Thus, both functions $\Phi$ and $\Psi$ can be extended on $[0,\overline{\theta}]$
and $\Psi$ is still continuous and strictly decreasing on $[0,\overline{\theta}]$.

\begin{lemma}\label{thirdequivalence}
The random walk is spectrally degenerate if and only if $\Psi(\overline{\theta})\geq 0$.
Moreover, $G'(R)$ is infinite if and only if $\Psi(\overline{\theta})\leq 0$.
\end{lemma}

\begin{proof}
This statement is a consequence of \cite[Theorem~9.22]{Woessbook}.

Assume first that $\Psi(\overline{\theta})<0$, hence by \cite[Theorem~9.22]{Woessbook}, it holds $\theta<\overline{\theta}$ and by Lemma~\ref{secondequivalence}, this implies that the random walk is not spectrally degenerate.
Moreover, in this case $\theta$ is the unique solution of $\Psi(t)=0$ in $(0,\overline{\theta})$.
In particular $\Psi(\theta)=0$, hence $G'(R)$ is infinite.

Assume now $\Psi(\overline{\theta})\geq 0$.
Then, \cite[Theorem~9.22]{Woessbook} implies $\theta=\overline{\theta}$ and so the random walk is spectrally degenerate.
On the one hand, if $\Psi(\overline{\theta})>0$, then $\Psi(\theta)>0$ and so $G'(R)$ is finite.
On the other hand, if $\Psi(\overline{\theta})=0$, then $\Psi(\theta)=0$ and so $G'(R)$ is infinite.
\end{proof}

Let us conclude this section by summarizing the situation as follows.

\begin{corollary}\label{trichotomy}
We have the following trichotomy.
\begin{itemize}
    \item If $\Psi(\overline{\theta})<0$, the random walk is spectrally non-degenerate and divergent.
    \item If $\Psi(\overline{\theta})=0$, the random walk is spectrally degenerate and divergent.
    \item If $\Psi(\overline{\theta})>0$, the random walk is spectrally degenerate and convergent.
\end{itemize}
\end{corollary}

\section{A divergent not spectrally positive recurrent random walk}\label{sectionconstruction}
In this section, we construct an adapted random walk on $\Gamma=\Z^3*\Z^d$, $d=5$ or 6,  which is divergent but not spectrally positive recurrent.
Such a random walk is necessarily spectrally degenerate and corresponds to the second case in Corollary~\ref{trichotomy}.
%This disproves a statement that appeared in a first version of \cite{DPT} and a similar statement in \cite{CandelleroGilch}, namely \cite[Lemma~4.5]{CandelleroGilch} as we explain below.

\subsection{Several erroneous lemmas}
We first restate \cite[Lemma~4.5]{CandelleroGilch} (switching the indices 1 and 2) and then explains how it leads to a contradiction.
This contradiction is what alerted us in the first place.
The flaw in the argument is quite subtle and we will come back to it in Section~\ref{sectionmistakes}.

\begin{wlemma}\label{lemma4.5}\cite[Lemma~4.5]{CandelleroGilch}
Assume that $\theta=\overline{\theta}=\theta_2/\alpha_2$ and that $G_1'(\zeta_1(R))$ and $G_2'(R_2)$ are finite, then if $\Psi(\overline{\theta})=0$, $\Phi''(\overline{\theta})$ is finite.
\end{wlemma}

By \cite[Proposition~1.9]{GouezelLalley}, the sum $I^{(1)}$ defined by~(\ref{defnI}) satisfies
\begin{equation}\label{formulafirstderivative}
    I^{(1)}(r)=rG'(r)+G(r).
\end{equation}
Similarly, letting for $i=1,2$
$$I_{G_i}^{(1)}(\zeta_i(r))=\sum_{x\in \Gamma_i}G_i(e,x|\zeta_i(r))G(x,e|\zeta_i(r)),$$
we have
\begin{equation}\label{iformulafirstderivative}
    I_{G_i}^{(1)}(\zeta_i(r))=\zeta_i(r)G_i'(\zeta_i(r))+G_i(\zeta_i(r)), i=1,2.
\end{equation}
By induction, similar formulae hold for higher derivatives, see \cite[Lemma~3.2]{DussauleLLT1}.
In the sequel, we will use in particular the one concerning the second derivatives :
\begin{equation}\label{formulasecondderivative}
    2rI^{(2)}(r)=2rG(r)+4r^2G'(r)+r^3G''(r).
\end{equation}

\medskip
First, by \cite[Proposition~6.3]{DG21}, the quantity
$$I^{(1)}_{i}=\sum_{x\in \Gamma_i}G(e,x|R)G(x,e|R)$$
is finite.
Using~(\ref{formulaWoess}), we deduce that $I^{(1)}_{G_i}(\zeta_i(r))$ is finite.
Consequently, for $i=1,2$, we have by~(\ref{iformulafirstderivative})
\begin{equation}\label{G'ifinite}
G_i'(\zeta_i(R))<\infty.
\end{equation}
Now, by Lemmas~\ref{firstequivalence} and~\ref{secondequivalence}, the equality $\theta=\theta_2/\alpha_2$ implies that the random walk is spectrally degenerate along $\Gamma_2$ and so $\zeta_2(R)=R_2$.
Finally,  $G_1'(\zeta_1(R))$ and $G_2'(R_2)$ are both finite.

\medskip
Second,
we need an explicit form of $\Phi''$.
The equality
$$G(r)=\Phi(rG(r))$$ readily implies
$$G'(r)=(rG(r))'\Phi'(rG(r))=(rG'(r)+G(r))\Phi'(rG(r))$$
and
\begin{align*}
G''(r)&=(rG(r))''\Phi'(rG(r))+((rG(r))')^2\Phi''(rG(r))\\
&=\frac{rG'(r)G''(r)+2G'(r)^2}{rG'(r)+G(r)}+(rG'(r)+G(r))^2\Phi''(rG(r)).
\end{align*}
Consequently,
\begin{equation}\label{equationphi''}
\Phi''(rG(r))=\frac{G(r)G''(r)-2G'(r)^2}{(rG'(r)+G(r))^3}.
\end{equation}
In particular, for $r=R$, we obtain
$$\Phi''(\theta)=\frac{G(R)G''(R)-2G'(R)^2}{(RG'(R)+G(R))^3}.$$
Therefore, if $G'(R)$ is finite, then $\Phi''(\theta)$ is finite if and only if $G''(R)$ is finite.

Now, differentiating twice Equation~(\ref{equationPhi}), we get
\begin{equation}\label{equationPhi''phii}
\Phi''(t)=\alpha_1^2\Phi''(\alpha_1t)+\alpha_2^2\Phi''(\alpha_2t),
\end{equation}
which implies that $\Phi''(\theta)$ is finite if and only if $\Phi_1''(\alpha_1\theta)$ and $\Phi_2''(\alpha_2\theta)$ are both finite.

Finally, by Corollary~\ref{trichotomy}, the condition $\Psi(\overline{\theta})=0$ is equivalent to the fact that the random walk driven by $\mu$ is spectraly degenerate and divergent.
Thus, Erroneous Lemma~\ref{lemma4.5} can be written as follows.

\begin{wlemma}[Alternative version of Erroneous Lemma~\ref{lemma4.5}]
Assume the random walk driven by $\mu$ is spectrally degenerate along $\Gamma_2$.
If $G'(R)$ is infinite, then $G_2''(R_2)$ is finite.
\end{wlemma}

A more general statement also appeared in a first version of \cite{DPT}, which led the authors to modify their statement.
Define for $i=1,2$,
$$I_{G_i}^{(2)}(\zeta_i(r))=\sum_{x,y\in \Gamma_i}G_i(e,x|\zeta_i(r))G(x,y|\zeta_i(r))G(y,e|\zeta_i(r)).$$
By~(\ref{formulaWoess}),
the quantity $J^{(2)}$ defined in~(\ref{defnJ}) can be written as
\begin{equation}\label{CorrespondenceJ2I2}
G(r)^3J^{(2)}(r)=G_1(\zeta_1(r))^3I_{G_1}^{(2)}(\zeta_1(r))+G_2(\zeta_2(r))^3I_{G_2}^{(2)}(\zeta_2(r)).
\end{equation}
In particular, $J^{(2)}(R)$ is finite if and only if
$I_{G_1}^{(2)}(\zeta_1(R))$ and $I_{G_2}^{(2)}(\zeta_2(R))$ are both finite.
Applying~(\ref{formulasecondderivative}) to $G_i$,
we have that
\begin{equation}\label{CorrespondenceJ2I2'}
    2\zeta_i(r)I_{G_i}^{(2)}(\zeta_i(r))=\zeta_i(r)G_i(\zeta_i(r))+4\zeta_i(r)^2G'_i(\zeta_i(r))+\zeta_i(r)^3G''_i(\zeta_i(r)).
\end{equation}
Since $G_i(\zeta_i(R))$ and $G'_i(\zeta_i(R))$ must be finite by~(\ref{G'ifinite}), we deduce that
\begin{equation}\label{CorrespondenceJ2G''}
    J^{(2)}(R)<\infty \text{ iff } G''_i(\zeta_i(R))<\infty, i=1,2.
\end{equation}
Thus, this lemma is a special case of the following wrong statement that appeared in a first version of \cite{DPT}.

\begin{wlemma}[Generalized version of Erroneous Lemma~\ref{lemma4.5}]\label{wronglemmaDPT}
In the context of relatively hyperbolic groups with respect to virtually abelian subgroups,
if $G'(R)$ is infinite, then $J^{(2)}(R)$ is finite, i.e.\ the random walk is spectrally positive recurrent.
\end{wlemma}

\subsection{Constructing counterexamples}

Let us disprove Erroneous Lemma~\ref{lemma4.5}.
From now on, we set $\alpha_1=\alpha$ and $\alpha_2=1-\alpha$.
First, note that there exists $\alpha_c$ such that
\begin{itemize}
    \item if $\alpha<\alpha_c$, then
$\overline{\theta}=\theta_2/\alpha_2<\theta_1/\alpha_1$,
\item if $\alpha=\alpha_c$, then $\overline{\theta}=\theta_1/\alpha_1=\theta_2/\alpha_2$,
\item if $\alpha>\alpha_c$, then $\overline{\theta}=\theta_1/\alpha_1<\theta_2/\alpha_2$.
\end{itemize}

Therefore,
\begin{itemize}
    \item if $\alpha<\alpha_c$, then  $\frac{\alpha}{1-\alpha}\theta_2<\theta_1$ and
$\Psi(\overline{\theta})=\Psi_1\left(\frac{\alpha}{1-\alpha}\theta_2\right )+\Psi_2(\theta_2)-1$.
\item if $\alpha=\alpha_c$, then $\Psi(\overline{\theta})=\Psi_1(\theta_1)+\Psi_2(\theta_2)-1$.
\item if $\alpha>\alpha_c$, then $\frac{1-\alpha}{\alpha}\theta_1<\theta_2$ and
$\Psi(\overline{\theta})=\Psi_1(\theta_1)+\Psi_2\left( \frac{1-\alpha}{\alpha}\theta_1\right )-1$.
\end{itemize}
As a consequence, the function $\alpha\mapsto \Psi(\overline{\theta})$ is continuous, see also \cite[Lemma~7.1]{CandelleroGilch}.

\medskip
We now set $\Gamma=\Z^{d_1}*\Z^{d_2}$
and we consider symmetric admissible and finitely supported probability measures $\mu_i$ on $\Gamma_i=\Z^{d_i}$, $i=1,2$.
It is well known that we have
$\mu_i^{*2n}(e)\sim C_iR_1^{-2n}n^{-d_i/2}$, see for instance \cite[Theorem~13.12]{Woessbook}.
Now we choose $d_1$ and $d_2$ in such a way that
\begin{itemize}
    \item $G_i(R_i)$ is finite, $i=1,2$,
    \item $G_1'(R_1)$ is infinite but $G_2'(R_2)$ is finite,
    \item $G''(R_2)$ is infinite.
\end{itemize}
These three conditions, together with the fact that $R_1=R_2=1$ impose that $d_1=3$ or 4 and $d_2=5$ or 6.
From now on, we set $d_1=3$ and we write $d=d_2\in \{5,6\}$.
In terms of the functions $\Phi_i$ and $\Psi_i$, it holds $\Psi_1(\theta_1)=0$, $\Psi_2(\theta_2)>0$ and $\Phi_2''(\theta_2)$ is infinite.

It follows that $\Psi(\overline{\theta})=0$ when $\alpha=1$, that $\Psi(\overline{\theta})=\Psi_2(\theta_2)>0$ when $\alpha=0$
and that  $\Psi(\overline{\theta})=\Psi_2(\frac{1-\alpha}{\alpha}\theta_1)-1<0$ when $\alpha\in [\alpha_c,1)$.
Thus by continuity, there exists $\alpha_*\in (0,\alpha_c)$ such that $\Psi(\overline{\theta})=0$ when $\alpha=\alpha_*$.
This yields for this value $\alpha_*$ of the parameter $\alpha$
$$\theta=\overline{\theta}=\theta_2/\alpha_2<\theta_1/\alpha_1, \text{ with }\alpha_1=\alpha_* \text{ and }\alpha_2=1-\alpha_*.$$

In other words, the random walk driven by $\mu_{\alpha_*}$ is spectrally degenerate along $\Gamma_2=\Z^d$ but not along $\Gamma_1=\Z^3$.
As a consequence, $\zeta_1(R)<R_1$ and so $G'(\zeta_1(R))$ is finite.
The assumptions of Erroneous Lemma~\ref{lemma4.5} are hence satisfied, so it would imply that $\Phi''(\overline{\theta})$ is finite, so $\Phi_2''(\theta_2)$ is finite by~(\ref{equationPhi''phii}).
This is a contradiction, so we disproved Erroneous Lemma~\ref{lemma4.5}.

Notice that the probability measure $\mu_{\alpha_*}$ satisfies the following properties.
\begin{enumerate}
    \item the random walk is spectrally degenerate along $\Gamma_2=\Z^d$,
    \item the random walk is not spectrally degenerate along $\Gamma_1=\Z^3$,
    \item $\Psi(\overline{\theta})=0$, hence the random walk driven by $\mu_{\alpha_*}$ is divergent,
    \item $\Phi''_2(\theta_2)$ is infinite, i.e.\ $G_2''(\zeta_2(R))$ is infinite. Thus $J^{(2)}$ is infinite by~(\ref{CorrespondenceJ2G''}) and the random walk driven by $\mu_{\alpha_*}$ is not spectrally positive recurrent.
\end{enumerate}

If we assume that $\mu_1$ or $\mu_2$ is aperiodic, i.e.\ $\mu_1^{*n}(e)$ or $\mu_2^{*n}(e)$ is positive for large enough $n$,
then $\mu_\alpha$ is also aperiodic for every $\alpha$.
This can be obtained for instance assuming that $\mu_1(e)$ and $\mu_2(e)$ are positive, i.e.\ by considering lazy random walks on the free factors.
This ends the proof of Theorem~\ref{maintheoremconstruction}. \qed

\medskip
We thus exhibited a phase transition at $\alpha=\alpha_*$, where the sign of $\Psi(\overline{\theta})$ changes, so does the behavior of the random walk by Corollary~\ref{trichotomy}.
Moreover, the following holds.
\begin{itemize}
    \item When $\Psi(\overline{\theta})<0$, the random walk is spectrally non-degenerate and by \cite[Theorem~1.1]{DussauleLLT2},
    $$\mu^{*n}(e)\sim CR^{-n}n^{-3/2}.$$
    \item when $\Psi(\overline{\theta})>0$, the random walk is convergent, hence spectrally degenerate.
    By \cite[Proposition~6.1]{DG21}, it cannot be spectrally degenerate along $\Z^3$.
    In this case, it holds by \cite[Theorem~1.3]{DPT} $$\mu^{*n}(e)\sim CR^{-n}n^{-d/2}.$$
\end{itemize}
As claimed in the introduction, at the phase transition $\alpha=\alpha_*$, the local limit theorem has an again different form.
This is the purpose of Section~\ref{SectionLLT}.

\subsection{Identifying the mistakes in Erroneous Lemmas}\label{sectionmistakes}
The mistake in the former version of \cite{DPT} when proving Erroneous Lemma~\ref{wronglemmaDPT} was to assume that the spectral radius $\rho_{H,r}$ of the first return transition kernel $p_{H,r}$ defined in~(\ref{defnfirstreturn}) were differentiable at $r=R$.
However, this differentiability property is only proved for convergent random walks.

The issue in \cite{CandelleroGilch} is more subtle.
The authors write $\zeta_i(r)=\zeta_i(R)+X_i(r)$ and first find a linear system of the form
$$C_1^{(i)}X_1(r)+C_2^{(i)}X_2(r)+o\big (R-r\big)=\mathrm{LP}_i(r),$$
$i=1,2$, where $\mathrm{LP}_i$ is a linear polynomial function.
Then, they derive a contradiction from this linear system, using the assumptions of Erroneous Lemma~\ref{lemma4.5}.
On Page~19 of \cite{CandelleroGilch}, they expand $(\zeta_i(R)+X_i(r))^n$ and then switch two sums to identify the coefficients $C_j^{(i)}$, see precisely \cite[(4.8)]{CandelleroGilch}.
However, switching sums is not legitimate, because the coefficients in front of $X_j(r)^{k_j}X_i(r)^{k_i}$ involve successive derivatives of the Green function $G_j$ at $\zeta_j(R)$ and these successive derivatives can be infinite.
This is typically the case when assuming that $\Phi''(\overline{\theta})$ is infinite and $\theta=\theta_2/\alpha_2=\overline{\theta}$, in which case the second derivative of $G_2$ at $\zeta_2(R)$ is infinite.

In any case, in both \cite{CandelleroGilch} and \cite{DPT}, the spotted invalid arguments are only related to the proofs of Erroneous Lemma~\ref{lemma4.5} and Erroneous Lemma~\ref{wronglemmaDPT} and do not affect the remainder of the papers.

\section{Local limit theorems}\label{SectionLLT}
We consider from now on the adapted probability measure $\mu_{\alpha_*}$ on $\Z^3*\Z^d$, with $d=5$ or $6$.
The random walk driven by $\mu_{\alpha_*}$ is spectrally degenerate along $\Z^d$, divergent, but not spectrally positive recurrent.
Now that $\alpha$ is fixed, we write $\mu=\mu_{\alpha_*}$ for simplicity.

For simplicity, we assume that $\mu_1$ and $\mu_2$ are aperiodic, i.e.\ $\mu_1^{*n}(e)$ and $\mu_2^{*n}(e)$ are positive for large enough $n$, so that $\mu$ is also aperiodic.
Our goal is to prove Theorem~\ref{maintheoremLLT}.

%Note that since we assume that the probability measures in consideration are all symmetric, the period of $\mu$ must be 1 or 2.
%By classical transformation, the same result holds for $\mu^{*2n}$ if we do not assume aperiodicity.

\subsection{Asymptotic differential equations}
By~(\ref{formulafirstderivative}) and~(\ref{formulasecondderivative}),  the two quantities $I^{(1)}(r)$ and $I^{(2)}(r)$ are related to the first and second derivatives of the Green function $G$.
Similarly, by~(\ref{CorrespondenceJ2I2}) and~(\ref{CorrespondenceJ2I2'}), the quantity $J^{(2)}$ is related to the second derivatives of the Green functions $G_i$, $i=1,2$.
One of the main results in \cite{DussauleLLT1} in the context of relatively hyperbolic group is the following rough formula that links the quantities $I^{(2)}$, $I^{(1)}$ and $J^{(2)}$ :
$$I^{(2)}(r)\asymp \left(I^{(1)}(r)\right)^3J^{(2)}(r)$$
which means that the ratio of these two quantities is bounded from above and below.

In the context of adapted measures on free products, the above rough estimates $\asymp$ can be improved to the more accurate asymptotics $\sim$ as follows.
\begin{proposition}\label{firstdifferential}
Consider an adapted probability measure $\mu_\alpha$ on $\Gamma=\Gamma_1*\Gamma_2$, with $0<\alpha<1$ and assume that $G'(R)=\infty$.
Then, there exist constants $C$, $c_1,c_2$ and $C'$ such that the following holds.
As $r\to R$, we have
$$G''(r)\sim C\left(G'(r)\right)^3\bigg(c_1G_1''\big(\zeta_1(r)\big)+c_2G_2''\big(\zeta_2(r)\big)-C'\bigg).$$
In particular, if $G''_1(\zeta_1(R))$ is finite and $G''_2(\zeta_2(R))$ is infinite, there exists $C$ such that
$$G''(r)\sim C\left(G'(r)\right)^3G_2''(\zeta_2(r)).$$
\end{proposition}

\begin{proof}
On the one hand, by~(\ref{equationphi''}) it holds
$$\Phi''(rG(r))=\frac{G(r)G''(r)-2G'(r)^2}{(rG'(r)+G(r))^3}.$$
The term $\frac{2G'(r)^2}{(rG'(r)+G(r))^3}$ converges to 0 as $r$ tends to $R$ and $G(r)$ converges to $G(R)$ which is finite.
Thus,
$$\Phi''(rG(r))\sim \frac{G(R)G''(r)}{R^3G'(r)^3}, r\to R.$$
On the other hand,
by~(\ref{equationPhi}) and Lemma~\ref{expressionzetaGzeta}
\begin{align*}
\Phi''(rG(r))&=\alpha_1^2\Phi_1''(\alpha_1rG(r))+\alpha_2^2\Phi_2''(\alpha_2rG(r))\\
&=\alpha_1^{2}\Phi_1''(\zeta_1(r)G_1(\zeta_1(r)))+\alpha_2^{2}\Phi_2''(\zeta_2(r)G_2(\zeta_2(r))).
\end{align*}
Therefore,~(\ref{equationphi''}) applied this time to the Green functions $G_i$ yields
$$\Phi''(rG(r))\sim c_1G_1''(\zeta_1(r))+c_2G_2''(\zeta_2(r))-C',$$
with
$$c_i=\alpha_i^2\frac{G_i(\zeta_i(R))}{(\zeta_i(R)G'_i(\zeta_i(R))+G_i(\zeta_i(R)))^3},i=1,2$$
and
$$C'=\frac{2\alpha_1^2G'_1(\zeta_1(R))^2}{\bigg(\zeta_1(R)G'_1(\zeta_1(R))+G_1(\zeta_1(R))\bigg)^3}+\frac{2\alpha_2^2G'_2(\zeta_2(R))^2}{\bigg(\zeta_2(R)G'_2(\zeta_2(R))+G_2(\zeta_2(R))\bigg)^3}.$$
This concludes the proof.
\end{proof}

We also prove the following result.
Recall that the quantities $I^{(k)}_i$ are defined in~(\ref{defnIki}).

\begin{lemma}\label{seconddifferential}
Consider an adapted probability measure $\mu_\alpha$ on $\Gamma=\Gamma_1*\Gamma_2$, with $0<\alpha<1$ and assume that $G'(R)=\infty$.
Then, there exists $C$ such that for $i=1,2$
$$\frac{d}{dr}\left (r^2I_{i}^{(1)}(r)\right )\sim C G'(r)I_{i}^{(2)}(r).$$
\end{lemma}

\begin{proof}
We write
$$r^2I_{i}^{(1)}(r)=\sum_{x\in \Gamma_i}rG(e,x|r)rG(x,e|r).$$
Using~(\ref{formulafirstderivative}) twice,
\begin{align*}
\frac{d}{dr}\left (r^2I_{i}^{(1)}(r)\right )=&\ r\sum_{x\in \Gamma_i}\sum_{y\in \Gamma}G(e,y|r)G(y,x|r)G(x,e|r)\\
&+r\sum_{x\in \Gamma_i}\sum_{y\in \Gamma}G(e,x|r)G(x,y|r)G(y,e|r).
\end{align*}
Fix $i\in \{1,2\}$ and $y\in \Gamma$.
Denote by $z$ the projection of $y$ on $\Gamma_i$.
In other words, $y$ may be written in its normal form as
$$y=zy_1y_2...y_n,$$
where $y_1\in \Gamma_1$ if $z\in \Gamma_2$ and conversely.
Since the random walk is adapted, it has to pass through $z$ before reaching $y$.
In other words,
$$\frac{G(e,y|r)}{G(e,e|r)}=\frac{G(e,z|r)}{G(e,e|r)}\frac{G(z,y|r)}{G(e,e|r)},$$
see also \cite[(3.3)]{WoessMartin}.
Therefore,
$$G(e,y|r)=\frac{1}{G(e,e|r)}G(e,z|r)G(z,y|r)$$
and similarly,
$$G(y,x|r)=\frac{1}{G(e,e|r)}G(y,z|r)G(z,x|r).$$
Thus, setting $\Gamma_z$ to be the set of $y\in \Gamma$ which project on $\Gamma_i$ at $z$, we get
$$\sum_{y\in \Gamma}G(e,y|r)G(y,x|r)=\frac{1}{G(e,e|r)^2}\sum_{z\in \Gamma_i}\sum_{y\in \Gamma_z}G(e,z|r)G(z,y|r)G(y,z|r)G(z,x|r)$$
with
$$\sum_{y\in \Gamma_z}G(z,y|r)G(y,z|r)=\sum_{y\in \Gamma_e}G(e,y|r)G(y,e|r)$$
by invariance by translation by $z$.
In particular, for $x=e$,
\begin{align*}
\sum_{y\in \Gamma}G(e,y|r)G(y,e|r)&=\frac{1}{G(e,e|r)^2}\sum_{z\in \Gamma_i}\sum_{y\in \Gamma_z}G(e,z|r)G(z,e|r)G(z,y|r)G(y,z|r)\\
&=\frac{1}{G(e,e|r)^2}\sum_{z\in \Gamma_i}\sum_{y\in \Gamma_e}G(e,z|r)G(z,e|r)G(e,y|r)G(e,z|r).
\end{align*}
As a consequence,
$$I^{(1)}(r)=\frac{1}{G(e,e|r)^2}\sum_{y\in \Gamma_e}G(e,y|r)G(y,e|r)I_{i}^{(1)}(r)$$
and so for all $z\in \Gamma_i$,
\begin{align*}
\frac{I^{(1)}(r)}{I_{i}^{(1)}(r)}&=\frac{1}{G(e,e|r)^2}\sum_{y\in \Gamma_e}G(e,y|r)G(y,e|r)\\
&=\frac{1}{G(e,e|r)^2}\sum_{y\in \Gamma_z}G(z,y|r)G(y,z|r).
\end{align*}
Combining all this, we get
$$\sum_{y\in \Gamma}G(e,y|r)G(y,x|r)=\frac{I^{(1)}(r)}{I_{i}^{(1)}(r)}\sum_{z\in \Gamma_i}G(e,z|r)G(z,x|r).$$
Similarly,
$$\sum_{y\in \Gamma}G(x,y|r)G(y,e|r)=\frac{I^{(1)}(r)}{I_{i}^{(1)}(r)}\sum_{z\in \Gamma_i}G(x,z|r)G(z,e|r).$$
Consequently,
\begin{align*}
\frac{d}{dr}\left (r^2I_{i}^{(1)}(r)\right )=&\ r\frac{I^{(1)}(r)}{I_{i}^{(1)}(r)}\sum_{x\in \Gamma_i}\sum_{z\in \Gamma_i}G(e,z|r)G(z,x|r)G(x,e|r)\\
&+r\frac{I^{(1)}(r)}{I_{i}^{(1)}(r)}\sum_{x\in \Gamma_i}\sum_{z\in \Gamma_i}G(e,x|r)G(x,z|r)G(z,e|r),
\end{align*}
which we rewrite
$$\frac{d}{dr}\left (r^2I_{i}^{(1)}(r)\right )=2r\frac{I^{(1)}(r)}{I_{i}^{(1)}(r)}I_{i}^{(2)}(r).$$
Since $G'(R)=\infty$,~(\ref{formulafirstderivative}) shows that
$I^{(1)}(r)\sim RG'(r)$ as $r\to R$.
Furthermore, by \cite[Proposition~6.3]{DG21}, the quantity $I_{i}^{(1)}(R)$ is finite.
Thus, we finally get
$$\frac{d}{dr}\left (r^2I_{i}^{(1)}(r)\right )\sim \frac{2R^2}{I_{i}^{(1)}(R)}G'(r)I_{i}^{(2)}(r)$$
as $r\to R$,
which concludes the proof.
\end{proof}

By combining Proposition~\ref{firstdifferential} and Lemma~\ref{seconddifferential}, we get the following statement.
\begin{corollary}\label{thirddifferential}
Assuming that $G'(R)$ is infinite, $G''_1(\zeta_1(R))$ is finite and $G''_2(\zeta_2(R))$ is infinite, there exists $C$ such that
$$\frac{G''(r)}{\left(G'(r)\right)^2}\sim C\frac{d}{dr}\left(r^2I_{2}^{(1)}(r)\right ).$$
\end{corollary}

\begin{proof}
By~(\ref{G'ifinite}), the quantity $G_2'(\zeta_2(R))$ is finite.
Applying~(\ref{formulasecondderivative}) to $G_2$, we get
$$G''_2(\zeta_2(r))\sim CI_{2}^{(2)}(r).$$
The result thus follows from Proposition~\ref{firstdifferential} and Lemma~\ref{seconddifferential}.
\end{proof}

We will also use the following result later on.
\begin{lemma}\label{correspondenceGreenandzeta}
We have
$$\zeta_2(R)-\zeta_2(r)\sim C\bigg(G_2(\zeta_2(R))-G_{2}(\zeta_2(r))\bigg).$$
\end{lemma}
\begin{proof}
By~(\ref{G'ifinite}), the quantity $G_2'(\zeta_2(R))$ is finite.
Derivating the Green function $G_2$ at $\zeta_2(R)$ yields
$$G_{2}(t)=G_{2}(\zeta_2(R))+G'_{2}(\zeta_2(R))(\zeta_2(R)-t)+o\left(\zeta_2(R)-t\right).$$
Applying this at $t=\zeta_2(r)$ gives the result.
\end{proof}

Everything is now settled to prove Theorem~\ref{maintheoremLLT}.
We treat separately the odd and even cases.

\subsection{The case $d=5$}
We consider the adapted probability measure $\mu_\alpha$ constructed in Section~\ref{sectionconstruction} and we set $\alpha=\alpha_*$ and write $\mu=\mu_\alpha$.
Recall that the measures $\mu_1$ and $\mu_2$ are assumed to be symmetric, admissible, aperiodic and finitely supported on $\Gamma_1=\Z^3$ and on $\Gamma_2=\Z^5$ respectively.
In particular, $R_1=R_2=1$ by \cite[Corollary~8.15]{Woessbook} and
\begin{itemize}
    \item  the random walk is not spectrally degenerate along $\Gamma_1$, so $G_1''(\zeta_1(R))$ is finite,
    \item it is spectrally degenerate along $\Gamma_2$, so $\zeta_2(R)=R_2==1$,
    \item $G_2(1)$ and $G_2'(1)$ are finite but $G_2''(1)$ is infinite,
    \item $G'(R)$ is infinite.
\end{itemize}

Moreover, the function $G''_2(t)$ has the following asymptotic expansion at $1$ :
\begin{equation}\label{asymptoticG''oddcase}
G''_{2}(t)\sim C_1\frac{1}{\sqrt{1-t}}, \text{ as }t\to 1.\ \footnote{This follows from the classical local limit theorem
$\mu_2^{*n}(e)\sim Cn^{-5/2}$
given for instance by \cite[Theorem~13.12]{Woessbook}
and from Karamata's Tauberian theorem \cite[Corollary~1.7.3]{BinghamGoldieTeugels}.
See also \cite[Proposition~17.16]{Woessbook} where the singular expansion at 1 of the Green function is given for simple random walks on $\Z^d$.}
\end{equation}

By applying Corollary~\ref{thirddifferential}, there exists $C_2>0$ such that as $r\to R$,
$$\frac{G''(r)}{\left(G'(r)\right)^2}\sim C_2\frac{d}{dr}\left(r^2I_{2}^{(1)}(r)\right ).$$
Integrating this asymptotic differential equation between $r$ and $R$ and using the fact that $G'(R)=\infty$ yields
\begin{equation}\label{twoterms}
\begin{split}
\frac{1}{G'(r)}\sim &\ C_2 R^2I_{2}^{(1)}(R)-C_2r^2I_{2}^{(1)}(r)\\&=C_2r^2\left (I_{2}^{(1)}(R)-I_{2}^{(1)}(r)\right)+C_2(R^2-r^2)I_{2}^{(1)}(R).
\end{split}
\end{equation}
Indeed,
$$
\frac{1}{G'(r)}=C_2 R^2I_{2}^{(1)}(R)-C_2r^2I_{2}^{(1)}(r)+\int_{r}^Ro\left (\frac{G''(\rho)}{\left(G'(\rho)\right)^2}\right)d\rho.
$$
For every $\epsilon>0$, there exists $r_0$ such that if $r\geq r_0$, we have
$$\left |o\left (\frac{G''(r)}{\left(G'(r)\right)^2}\right)\right |\leq \epsilon \left (\frac{G''(r)}{\left(G'(r)\right)^2}\right)$$
and so
$$\left |\int_{r}^Ro\left (\frac{G''(r)}{\left(G'(r)\right)^2}\right)d\rho\right |\leq \epsilon \frac{1}{G'(r)}.$$

By using~(\ref{formulafirstderivative}) for the Green function $G_{2}$, we get
$$\frac{d}{dt}(tG_{2}(t))=tG'_{2}(t)+G_{2}(t)=\sum_{x\in \Gamma_2}G_{2}(e,x|t)G_{2}(x,e|t).$$
For $t=\zeta_2(r)$, we get
$I_{2}^{(1)}(r)=\zeta_2(r)G'_{2}(\zeta_2(r))+G_{2}(\zeta_2(r))$.
Thus,
\begin{equation}\label{Threeterms}
\begin{split}
I^{(1)}_{2}(R)-I^{(1)}_{2}(r)= &\ \big(1-\zeta_2(r)\big)G'_{2}(1)+\big (G_{2}(1)-G_{2}(\zeta_2(r))\big )\\
&+\zeta_2(r)\left (G'_{2}(1)-G'_{2}(\zeta_2(r))\right),
\end{split}
\end{equation}
with
$$G'_{2}(1)-G'_{2}(\zeta_2(r))=\int_{\zeta_2(r)}^{1} G''_{2}(\rho)d\rho.$$
Therefore, by~(\ref{asymptoticG''oddcase}),
\begin{equation}\label{expansionG'zeta}
G'_{2}(1)-G'_{2}(\zeta_2(r))\sim C_3\sqrt{1-\zeta_2(r)}.
\end{equation}
According to Lemma~\ref{correspondenceGreenandzeta}, the two first terms in the right-hand side of~(\ref{Threeterms}) are of order of magnitude $1-\zeta_2(r)$, while the third one has order of magnitude
$\sqrt{1-\zeta_2(r)}$ by (\ref{expansionG'zeta}).
Consequently,
\begin{equation}\label{expansionG'zeta'}
I^{(1)}_{2}(R)-I^{(1)}_{2}(r)\sim C_4\sqrt{1-\zeta_2(r)}.
\end{equation}
By Lemma~\ref{expressionzetaGzeta}, $\zeta_2(r)G_{2}(\zeta(r))=\alpha_2rG(r)$, so
\begin{align*}
&\alpha_2(G(R)-G(r)) =\frac{\zeta_2(R)}{R}G_2(\zeta_2(R))-\frac{\zeta_2(r)}{r}G_2(\zeta_2(r))\\
&\hspace{.6cm}=\frac{\zeta_2(R)}{R}(G_2(\zeta_2(R))-G_2(\zeta_2(r)))+G_2(\zeta_2(r))\left (\frac{\zeta_2(R)}{R}-\frac{\zeta_2(r)}{r}\right)\\
&\hspace{.6cm}=\frac{\zeta_2(R)}{R}(G_2(\zeta_2(R))-G_2(\zeta_2(r)))+G_2(\zeta_2(r))\left (\frac{1}{R}-\frac{\zeta_2(r)}{r}\right).
\end{align*}
This readily implies
$$G(R)-G(r)\sim C_5(1-\zeta_2(r))+C_6(R-r).$$
Since $G'(R)$ is infinite, $(R-r)=o(G(R)-G(r))$,
hence
\begin{equation}\label{comparisonGG_2}
G(R)-G(r)\sim C_5(1-\zeta_2(r)).
\end{equation}
In particular, $G(R)-G(r)=o(1)$, so applying~\ref{expansionG'zeta'}, we get
$$I^{(1)}_{2}(R)-I^{(1)}_{2}(r)\sim C_7\sqrt{G(R)-G(r)}.$$
The equality $(R-r)=o(G(R)-G(r))$ also yields
$$R^2-r^2=o\left (\sqrt{G(R)-G(r)}\right ),$$
hence by~(\ref{twoterms}),
\begin{equation}\label{oddcasedifferentialequation}
\frac{1}{G'(r)}\sim C_8\sqrt{G(R)-G(r)}.
\end{equation}
Integrating between $r$ and $R$, we have
$$(G(R)-G(r))^{3/2}\sim C_{9}(R-r).$$
Re-injecting this in~(\ref{oddcasedifferentialequation}), we deduce that
\begin{equation}\label{finaldifferentialodd}
G'(r)\sim C_{10}\frac{1}{(R-r)^{1/3}}.
\end{equation}
We can then directly use \cite[Theorem~9.1]{GouezelLalley} to deduce that
$$\mu^{*n}(e)\sim C_{11}R^{-n}n^{-5/3}.$$
This concludes the proof of the case $d=5$ in Theorem~\ref{maintheoremLLT}. \qed

\subsection{The case $d=6$}
The proof for $d=6$ is very similar.
We still have that $R_1=R_2=\zeta_2(R)=1$, $G'(R)=\infty$, $G_1''(\zeta_1(R))<\infty$ and $G''_2(R_2)=\infty$.
We can thus apply Corollary~\ref{thirddifferential}, so that~(\ref{twoterms}) and~(\ref{Threeterms}) again holds in this situation.
Moreover,~(\ref{asymptoticG''oddcase}) is replaced with
\begin{equation}\label{asymptoticG''evencase}
G''_{2}(t)\sim -C_1\log (1-t), \text{ as }t\to 1, \ \footnote{As above, this follows from the classical local limit theorem
$\mu_2^{*n}(e)\sim Cn^{-3}$
given for instance by \cite[Theorem~13.12]{Woessbook}
and from Karamata's Tauberian theorem \cite[Corollary~1.7.3]{BinghamGoldieTeugels}.
See also \cite[Proposition~17.16]{Woessbook} where the singular expansion at 1 of the Green function is given for simple random walks on $\Z^d$.}
\end{equation}

As above, we integrate $G''_{2}$ between $\zeta_2(r)$ and $\zeta_2(R)=1$ to obtain
$$G'_{2}(1)-G'_{2}(\zeta_2(r))=\int_{\zeta_2(r)}^{1} G''_{2}(\rho)d\rho.$$
This time, using~(\ref{asymptoticG''evencase}) yields
$$G'_{2}(1)-G'_{2}(\zeta_2(r))\sim -C_2(1-\zeta_2(r))\log (1-\zeta_2(r)).$$
We deduce from~(\ref{Threeterms}) that
$$I^{(1)}_{2}(R)-I^{(1)}_{2}(r)\sim - C_3(1-\zeta_2(r))\log (1-\zeta_2(r)).$$
Applying~(\ref{comparisonGG_2}) and Lemma~\ref{correspondenceGreenandzeta}, we have
$$I^{(1)}_{2}(R)-I^{(1)}_{2}(r)\sim -C_4(G(R)-G(r))\log (G(R)-G(r)).$$
As above, the fact that $G'(R)$ is infinite yields $(R-r)=o(G(R)-G(r))$, with $G(R)-G(r)=o(1)$,
so
$$R^2-r^2=o\bigg ((G(R)-G(r))\log (G(R)-G(r))\bigg),$$
hence by~(\ref{twoterms}),
\begin{equation}\label{evencasedifferentialequation}
\frac{1}{G'(r)}\sim -C_5(G(R)-G(r))\log (G(R)-G(r)).
\end{equation}
Integrating between $r$ and $R$, we have
\begin{equation}
-(G(R)-G(r))^2\log (G(R)-G(r))\sim C_6(R-r),
\end{equation}
which we rewrite as
$$-(G(R)-G(r))\log (G(R)-G(r))\sim C_6\frac{R-r}{G(R)-G(r)}.$$
Re-injecting this in~(\ref{evencasedifferentialequation}), we deduce that
\begin{equation}\label{evencasedifferentialequationbis}
\frac{1}{G'(r)}\sim C_7\frac{R-r}{G(R)-G(r)}.
\end{equation}
Thus, by integration between $r$ and $R$,
\begin{equation}\label{evencasedifferentialequationter}
\log (G(R)-G(r))\sim C_8 \log (R-r).
\end{equation}
By multiplying~(\ref{evencasedifferentialequation}) and (\ref{evencasedifferentialequationbis}) and using (\ref{evencasedifferentialequationter}), we get
$$\frac{1}{G'(r)^2}\sim -C_{9}(R-r)\log (R-r),$$
i.e.
\begin{equation}\label{finaldifferentialeven}
G'(r)\sim C_{10}\frac{1}{\sqrt{-(R-r)\log (R-r)}}.
\end{equation}

It remains to derive from~(\ref{finaldifferentialeven}) the asymptotic behavior of $\mu^{*n}(e)$. In the previous case when $d= 5$, the estimation~(\ref{finaldifferentialodd}) allowed us to apply directly \cite[Theorem~9.1]{GouezelLalley}, whose proof is based on a version of Karamata's Tauberian Theorem given in \cite{BinghamGoldieTeugels}. Due to the presence of the factor $\log(R-r)$, which does not appear in \cite{GouezelLalley}, we need to detail the proof.

We introduce the power series
$$A(s)=\sum_{n\geq 0}n\mu^{*n}(e)R^ns^n=(sR)G'(sR)$$
whose radius of convergence is 1.
By \cite[Corollary~9.4]{GouezelLalley}, there exists $\beta>0$ such that
\begin{equation}\label{qnmun}
\mu^{*n}(e)R^n=q_n+O\left (\mathrm{e}^{-\beta n}\right ),
\end{equation}
where $q_n$ is an increasing sequence.
It follows from~(\ref{finaldifferentialeven}) that as $s$ tends to 1,
$$\sum_{n\geq0}nq_ns^n\sim C_{11}\frac{1}{\sqrt{-(1-s)\log (1-s)}}.$$
Applying \cite[Corollary~1.7.3]{BinghamGoldieTeugels} to the slowly varying function $s\mapsto (\log s)^{-1/2}$, we get
$$nq_n\sim C_{12}n^{-1/2}\log (n)^{-1/2}.$$
Thus, applying~(\ref{qnmun}),
$$n\mu^{*n}(e)R^n\sim C_{12}n^{-1/2}\log (n)^{-1/2}$$
and so
$$\mu^{*n}(e)\sim C_{12}R^{-n}n^{-3/2}\log (n)^{-1/2}.$$
This concludes the proof of the case $d=6$ in Theorem~\ref{maintheoremLLT}. \qed

\bibliographystyle{plain}
\bibliography{phasetransition}

\end{document}